\documentclass[a4paper]{amsart}

\usepackage{amsmath,amsthm,amssymb}
\usepackage{float}
\usepackage{url}
\usepackage{tikz}
\usetikzlibrary{arrows,positioning,decorations,decorations.markings,matrix} 
\advance \textheight by -1 \topmargin
\advance \textheight by -1 \headheight
\advance \textheight by -1 \headsep
\advance \textheight by -2 \footskip
\vsize = \textheight
\hsize = \textwidth
\theoremstyle{plain}
\newtheorem{theorem}{Theorem}
\newtheorem{lemma}[theorem]{Lemma}
\newtheorem{prop}[theorem]{Proposition}
\newtheorem{proposition}[theorem]{Proposition}

\newtheorem{algorithm}[theorem]{Algorithm}
\theoremstyle{remark}
\newtheorem{remark}[theorem]{Remark}
\newtheorem{example}[theorem]{Example}
\numberwithin{theorem}{section}

\newcommand{\E}{\mathbb{E}}

\newcommand{\Q}{\mathbb{Q}}
\newcommand{\Z}{\mathbb{Z}}

\newcommand{\Lat}{\textnormal{Lat}}
\newcommand{\Tr}{\textnormal{Tr}}

\newcommand{\Aut}{\textnormal{Aut}}
\newcommand{\diag}{\textnormal{diag}}

\newcommand{\CL}{\textnormal{CL}}
\newcommand{\GL}{\textnormal{GL}}
\newcommand{\SO}{\textnormal{SO}}

\newcommand{\Kern}{\textnormal{Ker}}

\newcommand{\isom}{\cong}

\newcommand{\Mass}{\textnormal{Mass}}

\newcommand{\abs}[1]{\lvert #1 \rvert}
\newcommand{\fraka}{\mathfrak{a}}
\newcommand{\frakb}{\mathfrak{b}}

\newcommand{\calM}{\mathcal{M}}

\newcommand{\calS}{\mathcal{S}}

\newcommand{\frakp}{\mathfrak{p}}

\newcommand{\NN}{\mathfrak{N}}

\newcommand{\id}{\textnormal{id}}

\newcommand{\mykill}[1]{}

\begin{document}

\bibliographystyle{plain}

\title{Quaternary quadratic lattices  over number fields}
\author{Markus Kirschmer}
\email{markus.kirschmer@math.rwth-aachen.de}
\author{Gabriele Nebe}
\email{nebe@math.rwth-aachen.de}
\address{Lehrstuhl D f\"ur Mathematik, RWTH Aachen University, 52056 Aachen, Germany}

\begin{abstract}
We relate proper isometry classes of maximal lattices in a
totally definite quaternary quadratic space $(V,q)$ with trivial discriminant 
to certain equivalence classes of ideals in the quaternion algebra representing 
the Clifford invariant of $(V,q)$. 
This yields a good  algorithm to enumerate a system of representatives of proper 
isometry classes of lattices in genera of maximal lattices in $(V,q)$. 
\end{abstract}



\maketitle

\section{Introduction}

Small dimensional lattices over algebraic number fields $K$ 
have been related to ideals in \'etale $K$-algebras by various authors. 
In his Disquisitiones Arithmeticae Gau\ss \ 
relates proper isometry classes of binary lattices to ideal classes in 
quadratic extensions of $K$.
 For ternary quadratic 
 forms a similar relation between lattices and 
 quaternion orders has been investigated by Peters (\cite{Peters}) and Brzezinski (\cite{brzb}, \cite{brza})
 based on results from Eichler and Brandt, 
 for a functorial correspondence see Voight (\cite{Voight}). 
 
Quaternary lattices have been investigated by Ponomarev
(\cite{p1,p3,p2,p4}), who relates the
 proper isometry classes of lattices in a 
 quaternary quadratic space to certain equivalence classes of ideals
 in a quaternion algebra, where he is particularly interested in the case where $K=\Q $.
 The present paper generalises Ponomarev's results to arbitrary totally 
 real number fields $K$ and develops a fast algorithm to 
 enumerate proper isometry classes in certain genera of 
 quaternary lattices. 

 To state our results let $(V,q)$ be  a totally definite quaternary quadratic space over $K$ 
 of square discriminant and  let
 $Q$ be the totally definite quaternion algebra representing its Clifford invariant.
 Then Theorem \ref{main} shows that 
the proper isometry classes of $\fraka $-maximal lattices 
 in $(V,q)$ are in bijection with certain equivalence classes of normal ideals 
 in $Q$ of norm $\fraka $.  
 This correspondence is used 
 to relate  the  mass formulas 
 of Siegel and Eichler  in Section \ref{Siegel}.
 Section 7 develops an algorithm to enumerate
a system of representatives
 of proper isometry classes of $\fraka $-maximal lattices in $(V,q)$ 
    based on the
   method of \cite{KirschmerVoight}. Algorithm \ref{algo} 
   is much more efficient than the 
   usual Kneser neighbour method (see for instance \cite{ScharlauHemkemeier} 
   for a description of a good implementation of this method).
     This is illustrated in a small and a somewhat larger  example in the end of the paper. 
     A further application to the classification of binary Hermitian lattices is given in 
     \cite{BinaryHerm}. 

{\bf Acknowledgements} 
The authors thank the anonymous referee for many helpful comments largely improving the
exposition of the results. 
The research is supported by the DFG within the framework of the  SFB TRR 195.

\section{Quadratic lattices over number fields} \label{quadraticlat}

In this section, we set up basic notation for quadratic lattices.
Let $K$ be a  number field and let $(V,q)$ be a non-degenerate quadratic space over $K$. 
The most important invariants of $(V,q)$ are the Clifford invariant $c(V,q) $ as defined in \cite[Remark 2.12]{Scharlau} and 
the determinant $\det(V, q)$, which is the square class of the determinant of a Gram matrix of $(V,q)$.
The interest in these two isometry invariants of quadratic spaces is mainly due to the following classical result by  Helmut Hasse.

\begin{theorem}[\protect{\cite{hasse}}] \label{hasse}
Over a number field $K$ the isometry class of a quadratic space is uniquely determined by its dimension, its determinant, its Clifford invariant and its signature at all real places of $K$.
\end{theorem}

Let $\Z_K$ be the ring of integers in $K$. 
A $\Z_K$-lattice $L$ in $(V,q)$ is a finitely generated $\Z_K$-submodule of $V$ that 
contains a $K$-basis of $V$. 
The orthogonal group
$$O(V,q) 
 := \{ \varphi \in \GL (V) \mid 
q(\varphi (v) ) = q(v) \mbox{ for all } v\in V \} $$
and its normal subgroup
$\SO(V,q) := \{ \varphi \in O(V,q)\mid \det (\varphi ) = 1 \} $
of proper isometries act on the set of all lattices in $(V,q)$. 
We call two lattices $(L,q) $ and $(L',q)$ in $(V,q)$ {\em properly
isometric}, $(L,q) \cong ^+ (L',q) $, if they are in the same orbit under
the action of $\SO(V,q) $ and denote by
 $$[(L,q)]^+ = [L]^+ = \SO (V,q) \cdot L $$ the
proper isometry class of the $\Z_K$-lattice $L$.
The stabiliser of $(L,q)$ in $\SO(V,q)$ is 
called the proper isometry group $\Aut^+(L,q)$ of $(L,q)$.
If we refer to the coarser notion of isometry and orbits under the full orthogonal
group, then the superscript $\ ^+$ is omitted. 

 Certain invariants of a $\Z_K$-lattice $(L,q)$ can be read off from the 
 transfer to the corresponding $\Z $-lattice  $(L,\Tr(q))$, 
 where 
       $$\Tr(q) : L \to \Q , \ell \mapsto \Tr _{K/\Q }(q(\ell )) .$$
The $\Z $-lattice $(L,\Tr(q))$ is called the \emph{trace lattice} of $(L,q)$

Given a place $\frakp$ of $K$, let $K_\frakp$ and $V_\frakp:= V \otimes_K K_\frakp$ be the completions of $K$ and $V$ at $\frakp$.
If $\frakp$ is finite, we denote by $\Z_{K_\frakp}$ and $L_{\frakp } := L \otimes _{\Z_K} \Z_{K_{\frakp}} $ the completions of $\Z_K$ and $L$ at $\frakp $.
Two lattices $(L,q)$ and $(L', q)$ in $(V, q)$ are \emph{in the same genus}, 
if 
$$(L_{\frakp } , q) \cong (L'_{\frakp }, q) 
 \mbox{ for all maximal ideals } \frakp \mbox{ of } \Z _K  .$$

The classification of all (proper) isometry classes of lattices in a given genus 
is an interesting and intensively studied problem 
 (see \cite{ScharlauHemkemeier,Schiemann,Kirschmerhabil}). 
One strategy is to embed an integral quadratic lattice $(L,q)$ into a maximal 
one and deduce the classification of the genus of $(L,q)$ from the 
one of maximal lattices. 
Recall that for a (fractional) ideal ${\mathfrak a}$ of
$\Z_K$ a lattice $(L,q)$ in $(V,q)$ is ${\mathfrak a}$-{\em maximal},
if $q(L) \subseteq {\mathfrak a}$ and $q(L') \not\subseteq {\mathfrak a}$ for all $\Z_K$-lattices $(L', q)$ in $(V,q)$ with $L \subsetneq L'$.
The $\Z_K$-maximal lattices are also called {\em maximal}.
Locally, all $\fraka$-maximal lattices are isometric, see \cite[Theorem 91:2]{OMeara}.
Hence the set of all  ${\mathfrak a}$-maximal lattices in $(V, q)$ forms
 a single genus, which we denote by ${\mathcal G}_{{\mathfrak a}} (V,q) $. 

The number of isometry classes in a genus is always finite and it is called the class number of the genus.
By the strong approximation theorem, see for instance \cite[Theorem 104:4]{OMeara},
the class number of a genus can be determined by local invariants if there is an infinite place $\sigma$ of $K$ such that $(V_\sigma, q)$ is isotropic.
So the only interesting case is when $K$ is totally real and $(V_\sigma,q)$ is definite for all infinite places $\sigma$ of $K$.
After rescaling, we assume that $(V, q)$ is {\em totally positive definite}, which means that $(V_\sigma, q)$ is positive definite for all these $\sigma$.
An element $a$ of the totally real number field $K$ is called {\em totally positive}, if $\sigma(a) > 0$ for all  infinite places $\sigma$ of $K$.

\section{Some basic facts about quaternion algebras} 

This section relates normal ideals in quaternion algebras to maximal lattices. 
A detailed discussion of the arithmetic of quaternion algebras 
can be found in \cite{EichlerZT}, \cite{Vigneras}, and \cite{Reiner}.
Let $Q$ be a totally definite quaternion algebra over an algebraic number field $K$.
Then $K$ is totally real and $Q$ has a basis 
$(1,i,j,ij)$ with $ij=-ji $ and $i^2=-a$, $j^2=-b$ for some totally positive $a,b\in K$.
The algebra $Q$ is also denoted by $Q= \left(\frac{-a,-b}{K} \right)$. 
It carries a canonical involution, $\overline{\phantom{x}} \colon Q\to Q$ 
defined by $\overline{t+xi+yj+zij} = t-xi-yj-zij $.
The reduced norm
$$n \colon Q\to K, n(\alpha ) = \alpha \overline{\alpha}$$ 
of $Q$ is a quaternary positive definite quadratic form over $K$ such that
$n(\alpha \beta ) = n(\alpha ) n(\beta ) $ for all $\alpha ,\beta \in Q$. 
The group of proper isometries of the quadratic space $(Q, n)$ is
\begin{equation*}
\SO(Q,n) = \{ x\mapsto \alpha x\beta \mid \alpha ,\beta \in Q^*,\  n(\alpha ) n(\beta) = 1 \}
\end{equation*} 
(see e.g. \cite[Appendix IV, Proposition 3]{Dieudonne} or \cite[Proposition 4.3]{NebeQuat}). 

The canonical involution $\overline{\phantom{x}}$  of $Q$ 
is an improper isometry of 
$(Q,n)$, so the full orthogonal group $O(Q,n)$ is generated by the normal subgroup $\SO(Q,n)$ and the canonical involution $\overline{\phantom{x}}$.

\begin{remark} 
	The  Gram matrix of $(Q,n)$ with respect to the basis $(1,i,j,ij)$ from above is 
$\diag (1,a,b,ab) $. 
Hence the determinant of $(Q,n)$ is a square and its Clifford invariant 
can be computed with \cite[Formula (11.12)]{Kneser} as 
the class of $Q$ in the Brauer group of $K$. 
\end{remark}

An {\em order} in $Q$ is a $\Z_K$-lattice that is a subring of $Q$. 
An order ${\mathcal M}$ is called {\em maximal} if it is not contained in any other order.

\begin{proposition}\label{maximalmaximal}
If $\mathcal{M}$ is a maximal order in $Q$, then $(\mathcal{M}, n) \in \mathcal{G}_{\Z_K}(Q,n)$ is 
a maximal lattice.
\end{proposition}

\begin{proof}
It is enough to show that for all prime ideals $\frakp $ of $\Z_K$ the completion 
$(\mathcal{M}_{\frakp } , n) $ is a maximal lattice in $(Q_{\frakp} , n)$. 
If $\frakp $ is not ramified, then $(\mathcal{M}_{\frakp } , n) $ is unimodular (see \cite[Theorem 20.3]{Reiner}) 
and if $\frakp $ is ramified in $Q$ then 
$\mathcal{M}_{\frakp } = \{ x\in  Q_{\frakp } \mid n(x) \in \Z _{K_{\frakp }} \} $ by \cite[Theorem 12.8]{Reiner}. 
In both cases the lattice $({\mathcal M}_{\frakp},n)$ is maximal.
\end{proof}

A $\Z_K $-lattice  $J$ in $Q$ is called 
{\em normal} if its \emph{right order}
$$O_{r }(J) := \{ \alpha \in Q \mid J\alpha  \subseteq J \} $$ is a maximal order in $Q$. 
Then also its \emph{left order}
$O_{\ell }(J) := \{ \alpha \in Q \mid \alpha J \subseteq J \} $ is maximal (see \cite[Theorem 21.2]{Reiner}) 
and $J$ is an invertible left (right) ideal of its left (right) order. 
Let $\mathcal{M}$ be a maximal order in $Q$. Then $J$ is called a \emph{two sided} ideal of ${\mathcal M}$, if $O_r(J) = O_{\ell }(J) = {\mathcal M}$.
The two sided ideals of $\mathcal{M}$ form an abelian group.
The \emph{normaliser} of ${\mathcal M}$ 
$$ N ({\mathcal M})  := \{ \alpha \in Q^* \mid \alpha  {\mathcal M} \alpha ^{-1} = {\mathcal M} \} $$
acts on this group by left multiplication. 
This action has finitely many orbits, the number of which is called the two sided 
ideal class number $H({\mathcal M})$ of ${\mathcal M}$.


\begin{remark}\label{improper}
	Any normal lattice $J_{\frakp }$ in the 
	completion $Q_{\frakp } = Q\otimes _K K_{\frakp}$ is free as a right $O_r(J_{\frakp})$-module and thus of the form 
	$\alpha  O_r(J_{\frakp}) $ for some $\alpha \in Q_{\frakp} ^*$. 
	The map 
\begin{align*}
	J_{\frakp} &\to J_{\frakp} \\
	\gamma &\mapsto \alpha  \overline{\gamma} \, \overline{\alpha}^{-1} 
\end{align*} 
	is an improper isometry of $(J_{\frakp} ,n)$.
\end{remark}

We call two normal lattices $I,J$ left, right, respectively two sided {\em equivalent}, 
if there are $\alpha ,\beta \in Q^*$ such that $I=\alpha J$, $I=J\beta  $, respectively $I=\alpha J\beta  $. 
We denote by 
$$C(J) := \{ \alpha J\beta  \mid \alpha ,\beta \in Q^* \} $$
the two sided equivalence class of the normal lattice $J$.

\begin{prop}\label{7.1}
  Let $I,J$ be normal lattices in $Q$.
\begin{enumerate}
  \item If $I$ and $J$ are two sided equivalent, then $O_r(I)$ and $O_r(J)$ are conjugate.
  \item Suppose $O_r(I) = O_r(J)$. Then $I$ and $J$ are two sided equivalent if and only if there exists some $\beta \in  N(O_r(I)) $ such that $I\beta $ is left equivalent to $J $. 
\end{enumerate}
\end{prop}
\begin{proof}
Suppose $I$ and $J$ are two sided equivalent. Then there exist $\alpha, \beta \in Q^*$ such that $\alpha I\beta  = J$.
Then $O_r(J) = O_r(\alpha I \beta) = \beta^{-1} O_r(I) \beta$ is conjugate to $O_r(I)$.
This shows the first assertion. Moreover, if $O_r(I) = O_r(J)$, then $\beta \in {N}(O_r(I))$. 
The converse of the second assertion is clear.
\end{proof}

The {\em norm} $n(J)$ of a lattice $J$ is the fractional ideal of $\Z_K$ generated by the norms of the
elements in $J$, 
$$n(J) := \sum _{\gamma \in J} \Z_K n(\gamma ) . $$
Clearly $n(\alpha J\beta ) = n(\alpha )n(\beta ) n(J) $ so the norm gives a well defined 
map 
$$\{ C(J) \mid J \mbox{ a normal lattice in } Q \} \to \CL^+(K), C(J) \mapsto [n(J)] $$  
from the set of equivalence classes of normal lattices in $Q$ into the 
narrow class group $\CL^+(K)$ of $K$. 

Let ${\mathfrak a} $ be a fractional ideal  of $\Z_K$.
We call a normal lattice $J$ in $Q$
of {\em type} $[{\mathfrak a}]$ if $[ n(J) ] = [{\mathfrak a}] $. 
This generalises the notion of stably free ideals,
which are the  normal lattices  of type $[\Z_K]$, see \cite[Section 35]{Reiner}.

\begin{proposition}\label{normalamaximal}
Let $J$ be a $\Z_K$-lattice in $Q$ and let $\fraka$ be a fractional ideal of $\Z_K$.
Then $J$ is a normal lattice in $Q$ with $n(J) = \fraka$ if and only if $(J, n)$ lies in $\mathcal{G}_\fraka(Q,n)$.
\end{proposition}
\begin{proof}
  Suppose first that $J$ is a normal lattice in $Q$ with $n(J) = \fraka$. 
  Then the right order $\mathcal{O}$ of $J$ is maximal.
  As in the proof of Proposition \ref{maximalmaximal} we pass to the completions and let 
  $\frakp $ be a maximal ideal of $\Z_K$. 
  As $J$ is locally free (see Remark \ref{improper}), there 
  exists $x_\frakp \in Q_\frakp^*$ such that $x_\frakp J_\frakp = \mathcal{O}_\frakp$.
  Assume that $(J_\frakp, n)$ is not $\fraka_\frakp$-maximal.
  Then there exists $y_\frakp \in Q_\frakp \setminus J_\frakp$ such that $n(J_\frakp + y_\frakp \Z_{K_\frakp}) \subseteq \fraka_{\frakp}$.
  But then $n(\mathcal{O}_\frakp + x_\frakp^{-1} y_\frakp \Z_{K_\frakp}) \subseteq \Z_{K_\frakp}$.
  This contradicts  Proposition \ref{maximalmaximal}. \\
  Suppose now that $(J, n)$ is an $\fraka$-maximal lattice in $(Q,n)$.
  Let $\mathcal{M}$ be some maximal order in $Q$.
  For each maximal ideal $\frakp$ of $\Z_K$ there exists some $z_\frakp \in Q_\frakp^*$ such that $n(z_\frakp) \Z_{K_{\frakp}}  = \fraka_{\frakp}^{-1}$.
  Then $(z_\frakp J_\frakp, n)$ is $\Z_{K_\frakp}$-maximal and by Proposition \ref{maximalmaximal}
  properly isometric to $(\mathcal{M}_\frakp, n)$.
  So there exist  $\alpha_\frakp, \beta_\frakp \in Q_\frakp^*$ with $n(\alpha_\frakp) = n(\beta_\frakp)$ such that $\alpha_\frakp  \mathcal{M}_\frakp \beta_\frakp^{-1} = z_{\frakp} J_\frakp$.
  Hence $ O_r(J_{\frakp}) = 
  O_r(  z_{\frakp}^{-1} \alpha_\frakp  \mathcal{M}_\frakp \beta_\frakp^{-1}) = \beta_\frakp \mathcal{M}_\frakp \beta_\frakp^{-1}$ is maximal.
  Thus $J$ is normal. As $n(\mathcal{M}_\frakp) = \Z_{K_{\frakp}} $ we conclude that 
  $n(J_{\frakp} ) = \fraka _{\frakp} $ for all maximal ideals $\frakp$, so $n(J) = \fraka $.
\end{proof}

 Let $J$ be a normal lattice of type $ [{\mathfrak a}]$. 
 Then $n(J) = a {\mathfrak a}$ for some
totally positive $a \in K$. 
By the theorem of Hasse-Schilling-Maass, there 
is some $\alpha \in Q$ such that $n(\alpha ) = \frac{1}{a }$. 
Then $n(\alpha J) = {\mathfrak a}$. 
So any  two sided equivalence 
class of type $[{\mathfrak a}]$  is represented by 
some normal lattice $J$ with $n(J) = {\mathfrak a}$. 
We call such a representative ${\mathfrak a}$-{\em normalised}.
Then the set of all ${\mathfrak a}$-normalised representatives of $C(J)$ is
$$
C_{\fraka }(J) := \{ \alpha  J \beta^{-1}  \mid \alpha ,\beta \in Q^*, n(\alpha )n(\beta^{-1} )\in\Z_K^* \},$$ 
the orbit of $J$ under the action of 
$\{ (\alpha , \beta ) \in Q^* \times Q^* \mid n(\alpha )  n(\beta^{-1} ) \in \Z_K^* \} . $
Let 
\begin{equation}\label{eq:NN}
\NN (J) := \{ ( \alpha , \beta ) \in  N(O_{\ell}(J)) \times  N(O_r(J)) \mid n(\alpha ) n(\beta^{-1} ) \in  \Z_K^* \} .
\end{equation}
Clearly $\NN(J)$ only depends on the left and right order of $J$.

\begin{lemma} \label{normaliserequal}
Let $J$ be a normal lattice and let $\alpha ,\beta  \in Q^*$. 
Then $\alpha J\beta^{-1} =J$ if and only if  $(\alpha,\beta ) \in \NN (J) $.
\end{lemma}

\begin{proof}
Suppose first that $\alpha J\beta^{-1} =J$. Then $O_r(J) = O_r(\alpha J\beta^{-1} ) = \beta  O_r(J) \beta ^{-1}$ and hence $\beta  \in   N(O_r(J)) $.
Similarly $\alpha \in  N(O_{\ell }(J))$. 
Moreover $n(J) = n(\alpha J\beta^{-1} ) = n(\alpha )n(\beta^{-1} ) n(J) $ implies that $n(\alpha )n(\beta^{-1} ) \in \Z_K^* $. 
\\
Suppose now $(\alpha,\beta ) \in \NN (J)$ and consider the ideal $I:=J^{-1}\alpha J\beta^{-1} $. 
  Then $O_{\ell} (I) = O_{\ell}(J^{-1}) = O_r(J )$ and $ O_r(I) = O_r (J \beta^{-1} ) = O_r(J)$ as $\beta \in {N}(O_r(J))$.
Hence $I$ is a two sided ideal of $O_r(J)$. 
    As $\alpha \in {N}(O_{\ell}(J)) $ and $O_{\ell}(J) = O_r(J^{-1})$, we have $n(I) = n(J^{-1}) n(\alpha) n(J) n (\beta^{-1}) = \Z_K $. 
    For every maximal ideal $\frakp$ of $\Z_K$ there exists a unique maximal two sided ideal of $\mathcal{M}$ containing $\frakp \mathcal{M}$ and these freely generate the group of all two sided ideals of $\mathcal{M}$, see \cite[Theorems 22.4 and 22.10]{Reiner}.
  So $I$ being a two sided ideal of $O_r(J) $ of norm $\Z_K$ implies that $I = O_r(J)$ and hence 
\[ J=JO_r(J) = J I = J J^{-1} \alpha J\beta^{-1}  = \alpha  J \beta^{-1} . \qedhere \]
\end{proof}

For a normal lattice  $J$ we set
$$U(J) := \{ n(\alpha )n(\beta ^{-1}) \mid (\alpha , \beta ) \in \NN (J) \} .$$
This is a subgroup of $\Z_K^*$ and since the norm of an element in $Q^*$ is always totally positive,
$U(J)$ is a subgroup of the group $\Z_{K,>0}^*$ of totally positive units of $\Z_K$.
It always contains $(\Z_K^*)^2 = \{ n(u) \mid u \in \Z_K^*\}$.

For each coset $u \in \Z_{K,>0}^*/U(J)$, we choose an element  $\alpha _u \in Q$ such that $n(\alpha _u) \in u$.

\begin{prop}\label{latJ}
Let $J$ be a normal lattice in $Q$ with $n(J) = \fraka $.
 A system of representatives of all proper isometry classes of lattices 
 $(I,n)$ where $I\in C_{\fraka }(J) $ is
$$\Lat(J) := 
\{ (\alpha _uJ,n) \mid u\in \Z_{K,>0}^*/U(J) \}  \:. $$
Moreover, 
  \[ \Aut^+(\alpha_u J, n) = \alpha_u \Aut^+(J,n) \alpha_u^{-1} \]
where 
$$\Aut^+ (J,n) = \{ \gamma \mapsto \alpha \gamma \beta^{-1}  \mid  
(\alpha , \beta ) \in \NN (J)  , n(\alpha ) = n(\beta )  \} .$$
\end{prop}

\begin{proof}
Let $I \in C_{\fraka} (J)$. Then there are $\alpha, \beta \in Q^*$ 
with $n(\alpha )n(\beta^{-1} ) \in \Z_K^*$ 
such that $I = \alpha J\beta ^{-1} $. 
Let $u := n(\alpha )n(\beta^{-1} ) U(J) \in \Z _{K,>0}^*/U(J)$.  
By definition there are $\alpha ' \in N(O_{\ell }(J)) , \beta' \in  N(O_r(J)) $  
such that $n(\alpha _u)  n(\alpha ' )n(\beta'^{-1}) = n(\alpha )n(\beta ^{-1}) $. 
As $n(\alpha_u)$ and $ n(\alpha )n(\beta ^{-1}) $ lie in $\Z_K^*$ we also have $n(\alpha ' )n(\beta'^{-1})\in \Z_K^*$.
So $(\alpha' , \beta ') \in \NN (J)$ and thus $\alpha ' J\beta'^{-1}=J$ by Lemma \ref{normaliserequal}.
Moreover $n(\alpha _u\alpha '  \alpha ^{-1}) = n(\beta  \beta'^{-1}) $, hence 
\[ (I,n)  = (\alpha J\beta ^{-1} ,n) \cong^+ (\alpha _u \alpha '  \alpha ^{-1} (\alpha J\beta ^{-1} ) \beta \beta'^{-1},n ) = (\alpha _u \alpha ' J\beta'^{-1},n ) = (\alpha _u J ,n ). \]
It remains to shows that two different elements in $\Lat(J)$ do not represent the same proper isometry class.
To this end let $(\alpha _uJ,n)$ and $(\alpha _vJ,n)$ be  properly isometric elements of $\Lat(J)$.
Then there are $\alpha ,\beta \in Q^*$ such that  $n(\alpha )n(\beta^{-1} ) =1$ and
$\alpha  \alpha _u J \beta^{-1}  = \alpha _v J $. Then $(\alpha _v^{-1} \alpha  \alpha _u ) J \beta^{-1}  = J $. 
Lemma \ref{normaliserequal} shows that $ \alpha _v^{-1} \alpha \alpha_u \in N (O_{\ell }(J)) $ and $ \beta  \in N(O_r(J))$.
Moreover, $n(\alpha _v^{-1}) n(\alpha _u) = n(\alpha _v^{-1} \alpha \alpha_u)n(\beta^{-1} ) \in U(J)$. So $u = v$. 
\end{proof}


\section{Quaternary lattices} 

In this section we summarise the results of the previous section in the context of 
a totally positive definite 
quadratic space $(V,q)$ of dimension 4 over some totally real number field $K$.
 To apply the theory of the previous section, we assume that  $\det(V,q)$ is a square in $K^*$. 
 Then the Clifford invariant $c(V,q) = [Q]$ is the class of a totally
 definite quaternion algebra $Q$ in the Brauer group of $K$  and by Theorem \ref{hasse} we have 
 that $$(V,q) \cong (Q,n) .$$ 
So without loss of generality, we may assume that $(V,q) = (Q,n)$.
 If $K=\Q $ it is shown in \cite{p3} that
 the proper isometry classes of lattices in the 
 genus ${\mathcal G}_{\Z} (Q,n) $ of maximal lattices in $(Q,n)$ 
 correspond to two sided equivalence classes of normal lattices $J$ in $Q$.
To extend this correspondence 
to our more general situation
let ${\mathfrak a}$ be a fractional ideal of $\Z_K$ and 
choose ${\fraka} $-normalised lattices $J_1,\ldots , J_k$ in $Q$ such that the disjoint union
$$\bigsqcup _{i=1}^k C_{\fraka} (J_i) $$
is the set of all ${\fraka }$-normalised normal lattices in $Q $.
The easiest way to see that $k$ is finite is the combination of the following theorem and the finiteness of class numbers of genera.

\begin{theorem}\label{main}
$\bigsqcup _{i=1}^k \Lat(J_i) $ is a system of representatives 
	of the proper isometry classes of lattices in 
	${\mathcal G}_{{\mathfrak a}} (Q,n) $.
\end{theorem}

\begin{proof}
Let $(J,n) \in {\mathcal G}_{{\mathfrak a}} (Q,n) $.
Proposition \ref{normalamaximal} shows that $J$ is an ${\fraka }$-normalised normal lattice in $Q $.
The choice of $J_1,\dots,J_k$ implies that there exists a unique index $1 \le i \le k$ such that $J \in C_{\fraka} (J_i)$.
Proposition \ref{latJ} shows that $(J,n)$ is properly isometric to one and only one lattice in $\Lat(J_i)$.
\end{proof}

\section{Eichler's mass formula.} 

As above let $Q$ be a totally definite quaternion algebra over the 
totally real number field $K$. 
Denote by $\frakp _1,\ldots , \frakp _s $ the maximal ideals of $\Z_K$ that 
ramify in $Q$ (i.e. where the completion $Q_{\frakp _i} $ is a division algebra). 
Let $\mathcal{M}$ be a maximal order in $Q$.
The unit group index $[\mathcal{M}^*: \Z_K^*]$ is finite, see for example \cite[Th\'eor\`eme V.1.2]{Vigneras}.
Let
$${\mathcal I}({\mathcal M}) := \{ I_1,\dots,I_h \}$$ be a system of representatives of the left equivalence classes of right ideals of $\mathcal{M}$.
The number of these classes is finite and does not depend on the maximal order $\mathcal{M}$.
Hence $h$ is called the {\em class number} of $Q$
and it is always bigger or equal to the {\em type number} $t$ of $Q$, the number of conjugacy classes of maximal orders in $Q$, see for example \cite[Th\'eor\`eme III.5.4]{Vigneras} and the accompanying discussion.
The {\em mass} of ${\mathcal M}$  is
\[ \Mass(\mathcal{M}) := \sum_{i=1}^h [ O_{\ell}(I_i)^*  : \Z_K^*]^{-1}   \:. \]

\begin{theorem}[Eichler \protect{\cite{EichlerZT}}]\label{eichlermass}
	$$\Mass (\mathcal{M}) = 2^{1-[K:\Q ]} | \zeta_K(-1) | h_K \prod _{i=1}^s (|\Z_K/\frakp_i|-1) $$
	where $h_K$ is the class number of $K$.
\end{theorem} 

Let 
$h_K^+:=|\CL^+(K) | $ be the narrow class number of $K$ and fix some narrow class $[\fraka ] \in \CL^+(K)$. 
Then we define 
$${\mathcal I} ({\mathcal M} ,[\fraka ]): = \{ I \in {\mathcal I}({\mathcal M}) \mid [n(I)] = [\fraka ] \} $$
and 
\[ \Mass(\mathcal{M},[\fraka ]) := \sum_{I \in {\mathcal I} ({\mathcal M} ,[\fraka ])} [ O_{\ell}(I)^*  : \Z_K^*]^{-1}   .\]

\begin{theorem} 
  $ \Mass (\mathcal{M}) = h_K^+ \cdot \Mass(\mathcal{M},[\fraka ])$.
\end{theorem}
\begin{proof}
There exists some maximal order  $\mathcal{M}'$ in $Q$ such that $\mathcal{I}(\mathcal{M}, [ \fraka ]) = \mathcal{I}(\mathcal{M}', [\Z_K])$.
The discussion after \cite[Th\'eor\`eme 1]{Vignerasunits} shows that $\Mass(\mathcal{M},[\fraka ]) = \Mass(\mathcal{M}',[\Z_K ])$ does not depend on the maximal order $\mathcal{M}'$.
Hence $\Mass(\mathcal{M}, [\fraka]) = \Mass(\mathcal{M}, [\frakb])$ for all fractional ideals $\frakb$ of $\Z_K$ and therefore
\[ \Mass(\mathcal{M}) = \sum_{[\frakb] \in \CL^+(K)} \Mass(\mathcal{M},[\frakb]) = h_K^+ \cdot \Mass(\mathcal{M}, [\fraka]) \:. \qedhere \]
\end{proof}

\section{The Minkowski-Siegel mass formula}\label{Siegel}

In the spirit of our paper relating 
normal ideals in the quaternion algebra $Q$ to 
maximal lattices in $(Q,n)$ this section compares 
Eichler's mass formula for ideals to the well known Minkowski-Siegel mass formula 
for lattices. Whereas Eichler's formula involves the  class number $h_K$ of $K$, 
the Minkowski-Siegel formula does not. 
Our comparison below explains how the class number cancels out. 

The quotient of the narrow class number and the class number is  
$$|\Z_{K,>0}^* /(\Z_{K}^*)^2 | = \frac{h_K^+}{h_K} =: 2^u . $$
Let ${\mathcal P}(\Z_K) \cong K^*/\Z_K^* $ be the group of fractional 
principal ideals. 

Let ${\mathcal M}_1,\ldots, {\mathcal M}_t$ represent the conjugacy classes of 
maximal orders in $Q$ and let
$$N_i:=  N ({\mathcal M}_i)/K^* \quad (1\leq i \leq t).$$ 
For $1\leq i ,j \leq t $ we define the following maps: 
$$\overline{n_i} : N_i \to {\mathcal P}(\Z_K) / {\mathcal P}(\Z_K) ^2, 
\alpha K^* \mapsto (n(\alpha ))  {\mathcal P}(\Z_K) ^2 $$
and 
$$\overline{n_i} \times \overline{n_j} : N_i\times N_j \to {\mathcal P}(\Z_K) / {\mathcal P}(\Z_K) ^2,
(\alpha K^*,\beta K^*) \mapsto (n(\alpha )n(\beta ^{-1}))   {\mathcal P}(\Z_K) ^2. $$
Let $\Pi _i:= \overline{n_i} (N_i) $ be the image of $\overline{n_i}$ and $2^{f_i} := |\Pi _i |$ denote its order. 
The kernel of $\overline{n_i}$ is $\mathcal{M}_i^* K^*/K^* \isom \mathcal{M}_i^* / \Z_K^*$.
Thus $2^{f_i} = \abs{N_i} / [ \mathcal{M}_i^* : \Z_K^*]$.
By \cite[p. 137]{EichlerZT}, the order of the two sided ideal class group of ${\mathcal M}_i$ is 
$$H({\mathcal M}_i) = 2^{s-f_i} h_K . $$
Moreover the image of 
$\overline{n_i} \times \overline{n_j}  $ is $\Pi _i \Pi _j $ of order $2^{f_i+f_j-f_{ij} }$ 
where $$|\Pi _i \cap \Pi _j | = 2^{f_{ij}}.$$ 
Let
\begin{align*}   U_{ij}:= \Kern (\overline{n_i} \times \overline{n_j} )  & =  
\{ (\alpha K^* ,\beta K^* ) \in N_i \times N_j \mid n(\alpha \beta ^{-1}) (K^*)^2 \in \Z_K^* (K^*)^2 \} 
\\
& = \{ (\alpha K^* ,\beta K^* ) \mid (\alpha, \beta ) \in \NN ({\mathcal M}_i {\mathcal M}_j)\}  \end{align*} 
where $\NN ({\mathcal M}_i {\mathcal M}_j)$ is given by eq. \eqref{eq:NN} and define
	$$\overline{n_{ij}} :  U_{ij} \to {\mathcal P}(\Z_K) / {\mathcal P}(\Z_K) ^2 , 
	(\alpha K^* ,\beta K^* ) \mapsto (n(\alpha )) {\mathcal P}(\Z_K) ^2 .$$
	Then the image of $\overline{n_{ij}} $ is exactly $\Pi _i \cap \Pi _j $ 
	and the kernel of $\overline{n_{ij}} $ is 
	\begin{align*}
		\{ (\alpha K^* ,\beta K^* ) \in N_i \times N_j \mid n(\alpha ) \in \Z_K^* (K^*)^2, n(\beta ) \in \Z_K^* (K^*)^2 \}  \\
= {\mathcal M}_i^* K^* /K^* \times {\mathcal M}_j^* K^* /K^* \cong {\mathcal M}_i^*/\Z_K^* \times  {\mathcal M}_j^* / \Z_K^* .\end{align*} 
	We need one more map 
	$$\widetilde{n_{ij}}:  U_{ij} \to \Z_{K,>0}^* (K^*)^2/(K^*)^2 \cong \Z_{K,>0}^* /(\Z_{K}^*)^2,\ 
	(\alpha K^*,\beta K^*) \mapsto n(\alpha \beta ^{-1}) (K^*)^2 $$ and its kernel 
	$ V_{ij} := \Kern (\widetilde{n_{ij}}) $. 
	
	Let $\mathcal{M}_i^{(1)} = \{ \alpha \in \mathcal{M}_i^* \mid n(\alpha) = 1 \}$ be the norm one subgroup of $\mathcal{M}_i^*$.
	Since $Q$ is totally definite, the group $\mathcal{M}_i^{(1)} $ is finite.
	Let $2^{y_{ij}}$ be the index of $ {\mathcal M}_i^{(1)}  / \{\pm 1\} \times {\mathcal M}_j^{(1)} /\{\pm 1\}$ in $ V_{ij}$.
	
	\begin{remark}\label{zij} 
		Let $J$ be a normal lattice in $Q$ with right order $O_r(J) = {\mathcal M}_j$ and 
		left order $O_{\ell}(J) = {\mathcal M}_{i}$. 
		Then the subgroup $U(J)\leq \Z_{K,>0}^* $  from Proposition \ref{latJ} satisfies 
		$$U(J) / (\Z_K^*)^2  = \widetilde{n_{ij}} ( U_{ij}) .$$
		In particular 
		$$|\Z_{K,>0}^*/U(J)| =: 2^{z_{ij}} \mbox{ with } 2^{u-z_{ij}} = |U_{ij}/ V_{ij} | .$$
	\end{remark}

All the groups defined above contain 
$${\mathcal M}_i^{(1)} K^*/K^*
\times {\mathcal M}_j^{(1)} K^*/K^* \cong  
{\mathcal M}_i^{(1)}  / \{\pm 1\}
\times {\mathcal M}_j^{(1)} /\{\pm 1\} \:. $$  
For further computations we define 
$$ 2^{x_i} := |{\mathcal M}_i^*/{\mathcal M}_i^{(1)} \Z_K^* | = |n({\mathcal M}_i^*)/(\Z_K^*)^2| .$$
Figure \ref{fig:NiNj} illustrates the various subgroups of the group $N_i\times N_j$.

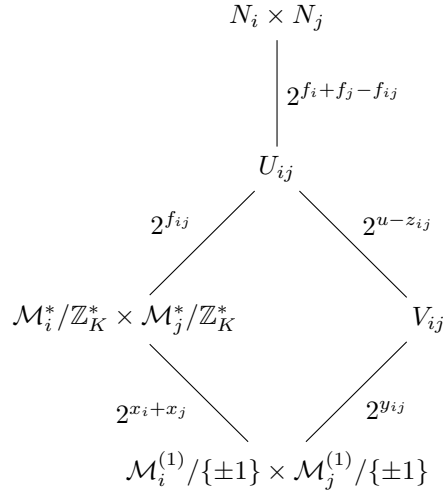
\begin{figure}[!h]
\begin{center}
\begin{tikzpicture}
\node (NN) at ( 0, 0) {$N_i \times N_j$};
\node (U)  at ( 0,-2) {$U_{ij}$};
\node (V)  at ( 2,-4) {$V_{ij}$};
\node (ME) at (-2,-4) {${\mathcal M}_i^*/\Z_K^* \times {\mathcal M}_j^*/ \Z_K^*$};
\node (M1) at ( 0,-6) {${\mathcal M}_i^{(1)} / \{\pm 1\} \times {\mathcal M}_j^{(1)} / \{\pm 1\}$}; 
\draw (NN) --node[right]{$2^{f_i+f_j-f_{ij}}$} (U) --node[above right]{$2^{u-z_{ij}}$} (V) --node[below right]{$2^{y_{ij}}$} (M1) --node[below left]{$2^{x_i+x_j}$} (ME) --node[above left]{$2^{f_{ij}}$} (U);
\end{tikzpicture}
\end{center}
\caption{Some subgroups of $N_i \times N_j$ and their indices.}\label{fig:NiNj}
\end{figure}

\begin{lemma}\label{autij}
	Let $J$ be as in Remark \ref{zij}. 
	Then the proper isometry group of the lattice $(J,n)$ only depends on the 
	left and right orders of $J$ and 
	$$|\Aut^+ (J,n) | = \frac{1}{2}
	|{\mathcal M}_i^{(1)}|
	|{\mathcal M}_j^{(1)}| 2^{y_{ij}} .$$
\end{lemma} 
\begin{proof}
By Proposition \ref{latJ}  every proper automorphism of $(J,n)$ is of the form $x \mapsto \alpha x \beta^{-1}$ with $\alpha \in {N}(\mathcal{M}_i)$, $\beta \in {N}(\mathcal{M}_j)$ and $n(\alpha) = n(\beta)$.
This induces an epimorphism $\Aut^+(J,n) \to  V_{i,j}$ with kernel $\{ \pm \id_J \}$.
\end{proof}


\begin{lemma}\label{numij}
The number of two sided equivalence classes represented by normal lattices in $Q$ having left order $\mathcal{M}_i$ and right order $\mathcal{M}_j$ is
$$h_K 2^{s-f_i-f_j+f_{ij} } .$$
\end{lemma}
\begin{proof}
Let $(T_1,\dots,T_{2^s h_K})$ be a transversal of $\{ x\calM_i \mid x \in K^* \}$ in the abelian group of two sided ideals of $\calM_i$.
We consider the set $\calS := \{ T_{\ell } \calM_i \calM_j \mid 1 \le {\ell } \le 2^s h_K\}$.
The group $N_i \times N_j$ acts on $\calS$ via
\[ (N_i \times N_j) \times \calS \to \calS,\; ((aK^*, bK^*), I) \mapsto a\lambda Ib^{-1}  \]
where $\lambda \in K^*$ is chosen such that $a\lambda Ib^{-1} \in \calS$.
Lemma \ref{normaliserequal} shows that the stabiliser of any ideal in $\calS$ is $U_{ij}$.
In particular, $\calS$ consists of $h_K 2^{s-f_i-f_j+f_{ij} }$ orbits.
The result follows since the number of orbits is also the number of two sided equivalence classes represented by normal lattices in $Q$ having left order $\mathcal{M}_i$ and right order $\mathcal{M}_j$.
\end{proof}

To state the Minkowski-Siegel mass formula let
 $L_1,\ldots , L_k$ be a system of representatives of proper isometry classes of lattices in 
${\mathcal G}_{\fraka } (Q,n) $. 
	Then the mass of this genus of ${\fraka}$-maximal lattices is defined as 
	$$\Mass ({\mathcal G}_{\fraka} (Q,n) ) := \sum _{i=1}^k \frac{1}{|\Aut^+(L_i)|} .$$ 
	Already Siegel gave an analytic expression for the 
	mass of a genus of arbitrary positive definite $\Z_K$-lattices (see \cite{SiegelI} and
	\cite{SiegelIII}). 
	In our special situation, this expression can also be derived from
 Eichler's mass formula: 

\begin{theorem} 
	For any fractional ideal $\fraka $ of $K$ 
	$$\Mass ({\mathcal G}_{\fraka } (Q,n) ) =
2^{1-2[K:\Q ]}  \zeta_K(-1) ^2 \prod _{i=1}^s \frac{(|\Z_K/\frakp_i|-1)^2}{2} .$$
\end{theorem} 

\begin{proof}
	Clearly the map $(L,n) \mapsto (L,an) $ is an isometry preserving  bijection between 
	${\mathcal G}_{\fraka} $ and ${\mathcal G}_{a\fraka} $ for any totally positive $a\in K$. 
	So it is enough to show the theorem for representatives 
	$\fraka _1$, $\ldots $, $\fraka _{h_K^+} $ of $\CL^+(K)$.

	 We fix an order ${\mathcal M}_j$  and some $1\leq i \leq t$.
 Remark \ref{numij} gives the number of right ideals in ${\mathcal I}({\mathcal M_j}) $ having left order 
	 isomorphic to ${\mathcal M}_{i}$ as  
	 $h_K 2^{s-f_i-f_j+f_{ij} } .$ By Proposition \ref{latJ} these right ideals 
	 give rise to $2^{z_{ij}}$ proper isometry classes of lattices (see Remark \ref{zij}), all having 
	 the same proper isometry group which has order 
	 $2 |{\mathcal M}_i^{(1)}/ \{\pm 1\} | |{\mathcal M}_j^{(1)}/ \{\pm 1\} | 2^{y_{ij}} $ 
	 by Lemma \ref{autij}.
	 So 
	 \begin{align*} 
		 \sum _{i=1}^{h_K^+} \Mass ({\mathcal G}_{\fraka _i} (Q,n) ) & =  
		 \sum _{i,j =1}^t 2^{z_{ij}} 
		 \frac{h_K 2^{s-f_i-f_j+f_{ij} } }{2 |{\mathcal M}_i^{(1)}/ \{\pm 1\} | |{\mathcal M}_j^{(1)}/ \{\pm 1\} | 2^{y_{ij}}}  \\ 
		 & = 
		 \sum_{i,j=1}^t h_K 2^{s-f_i} 2^{s-f_j} 2^{u-s} 
	 \frac{2^{f_{ij}}}{2 |{\mathcal M}_i^{(1)}/ \{\pm 1\} | |{\mathcal M}_j^{(1)}/ \{\pm 1\} | 2^{y_{ij}-z_{ij}+u}} \\
	 \intertext{ using $y_{ij}-z_{ij} +u = f_{ij}+x_i+x_j$ we conclude }
	 &=
	 \sum_{i,j=1}^t \frac{H({\mathcal M}_i) H({\mathcal M}_j) }{h_K} \frac{h_k^+}{h_K}  2^{-s} 
	 \frac{1}{2 |{\mathcal M}_i^{*}/\Z_K^* | |{\mathcal M}_j^{*}/\Z_K^* | }  \\
 &= \Mass({\mathcal M})^2 \frac{h_K^+}{h_K^2} 2^{-1-s} .\end{align*} 
	 Now by \cite{GHY}
	 the mass of ${\mathcal G}_{\fraka _i} (Q,n)$ does not depend on $i$, 
	(as locally the lattices are just rescaled versions of each other) so  for all $i$ 
	$$\Mass ({\mathcal G}_{\fraka _i} (Q,n) ) = \Mass({\mathcal M})^2 h_K^{-2} 2^{-1-s} $$ 
	and the theorem follows from the computation of $\Mass({\mathcal M}) $ in 
	Theorem \ref{eichlermass}.
\end{proof}

\section{Proper isometry classes in  ${\mathcal G}_{\fraka } (Q,n)$} \label{algq}

This section uses  the method from \cite{KirschmerVoight} to develop an algorithm 
for determining a system of representatives of the proper isometry classes 
in ${\mathcal G}_{\fraka }(Q,n) $. As explained in Remark \ref{schneller} below, this yields a much faster algorithm to enumerate this genus 
than the usual neighboring algorithm.

\begin{algorithm}\label{algo}
  Given a totally definite quaternion algebra $Q$ over $K$ and 
  a fractional ideal $\fraka$ of $\Z_K$, the following algorithm returns a system of representatives of the proper isometry classes in ${\mathcal G}_{\fraka} (Q, n)$.
  \begin{enumerate}
    \item Compute a maximal order $\mathcal{M}$ in $Q$ using Zassenhaus' Round 2 algorithm \cite{Round2} or Voight's specialised algorithm \cite[Algorithm 7.10]{Voight_IMR}. 
    \item Using \cite[Algorithm 7.10]{KirschmerVoight} compute:
      \begin{enumerate}
	\item A system of representatives $(\mathcal{M}_1,\dots,\mathcal{M}_t)$ of the conjugacy classes of maximal orders in $Q$.
	\item A system of representatives $(I_1,\dots,I_h)$ of all invertible right ideals of $\mathcal{M}$ up to left equivalence.
      \end{enumerate}
    \item For $1\le i \le t$ set $S_i := \{ I_j \mathcal{M}_i \mid 1 \le j \le h \mbox{ and } [n(I_j \mathcal{M}_i)] =  [\fraka] \}$.
    \item If $g \in {N}(\mathcal{M}_i) $ and $I \in S_i$ then there exists a unique lattice $J \in S_i$ such that $J$ is left equivalent to $I g^{-1}$.
      This yields  an action of the normaliser ${N}(\mathcal{M}_i)$ on $S_i $.
      For $1 \le i \le t$ compute a system of orbit representatives $T_i$ of this action.
    \item For $J \in \bigcup_i T_{i}$ fix some totally positive generator 
	$a_J$ of $n(J)^{-1}\fraka$ and compute some $x_J \in Q^*$ such that $n(x_J) = a_J$.
    \item For $u \in \Z_{K, >0}^*/(\Z_K^*)^2$ compute some $\alpha_u \in Q^*$ such that $n(\alpha_u) \in u$.
    \item Return $\{ (\alpha_u x_J J, n) \mid J \in \bigcup_i T_{i} \mbox{ and } u \in \Z_{K,>0}^*/U(J) \}$.
  \end{enumerate}
\end{algorithm}
	
\begin{proof}
  We only need to show that the output of the algorithm is correct.
  The set $\{ I_i \mathcal{M}_j \mid 1 \le i \le h \}$ is a system of representatives of the left equivalence classes of all invertible right ideals of $\mathcal{M}_j$.
  Thus Proposition \ref{7.1} shows that $\bigcup_i T_i$ is a system of representatives of the two sided equivalence classes of all normal lattices in $(V,q)$ of type $[\fraka]$.
  For any lattice $J \in T_i$, the class $[n(J)^{-1} \fraka] $ is trivial. Hence the scalar $a_J$ exists.
  The existence of the elements $x_J$ and $\alpha_u$ follows from the Theorem of Hasse-Schilling-Maass.
  Then $x_J J$ is $\fraka$-normalised. 
  Proposition \ref{latJ} shows that the proper isometry classes of lattices $(I,n)$ with $I\in C_{\fraka}(x_J J)$ is given by
            $$\Lat(x_J J) = \{ (\alpha_u x_J J, n) \mid u\in \Z_{K,>0}^*/U(J) \}  \:. $$
  Hence the set computed in (7) 
  is a system of representatives of the proper isometry classes of all $\fraka$-maximal lattices in $(Q, n)$.
\end{proof}

\begin{remark}We give some remarks concerning the last three steps in the previous algorithm.
\begin{enumerate}
\item
Let  $J,J' \in T_{i}$. Proposition \ref{latJ} shows that $U(J) = U(J')$ whenever the left orders of $J$ and $J'$ are conjugate.
This can be used to speed up the last step of the algorithm.
\item
  The norms of the ideals $J \in \bigcup_i T_i$ will only be supported by very few prime ideals.
  So for the computation of $x_J$ and $\alpha_u$ in steps (5) and (6) one only has to solve very few norm equations of the form
    \[ n(x) = a \quad \mbox{ with  $a \in K_{>0}$.} \]
    The Theorem of Hasse-Schilling-Maass (or the Hasse principle for quadratic forms) shows that any such norm equation has a solution $x \in Q^*$ and it gives rise to the isotropic vector $(1,x)$ of the quintic quadratic space $\langle -a \rangle \perp (Q, n)$.
  This is how it such a solution $x \in Q^*$ can be found.
\item
  Note that left multiplication with $(\alpha _u x_J)^{-1}$ gives an isometry between $ (\alpha _u x_J J,n)$ and $(J, u_0 a_J n ) $ where $u_0 = n(\alpha_u)$.
  So for most applications, it is not necessary to compute the elements $x_J$ and $\alpha_u$ in steps (5) and (6).
\end{enumerate}
\end{remark}

\begin{remark}\label{schneller}
	The computation of a system of representatives of the proper isometry classes in ${\mathcal G}_{\fraka }(Q,n) $ using Algorithm \ref{algo} 
	is much faster than using Kneser's neighbour method \cite{ScharlauHemkemeier} directly.
There are mainly two reasons for this.
\begin{enumerate}
\item
%
Let $s$ be the number of finite places of $K$ which ramify in $Q$ and suppose that $K$ has narrow class number $1$.
Section 4 of \cite{EichlerZT} shows that $H(\mathcal{M}_i) \leq 2^s$ and $t \leq h = \sum_{i= 1}^t H(\mathcal{M}_i) \le 2^s t$.
Moreover, $\lvert {N}(\mathcal{M}_i) / \mathcal{M}_i^* K^* \rvert \le 2^s$.
By Algorithm \ref{algo} the number of proper isometry classes in  ${\mathcal G}_{\fraka }(Q,n) $ is at least  $ht/2^s \ge h^2/2^{2s}$.
So using Kneser's method directly requires to enumerate way more lattices than the enumeration of the $h$ ideal classes in $\mathcal{M}$.
\item 
The bottleneck of Kneser's method is the computation of many isometries between $\Z_K$-lattices.
The computation of such an isometry is usually done by computing a suitable isometry of the corresponding trace lattices, see for example \cite[Remark 2.4.4]{Kirschmerhabil}.
Since the trace lattices have rank $4[K:\Q]$, this method is limited to $[K:\Q]$ being small.
\\
Computing a system of representatives for the right ideal classes of $\mathcal{M}$ does not require the computations of isometries, since isomorphism tests for normal ideals amount only to show that 
a certain $\Z$-lattice has minimum $\leq [K:\Q ]=\Tr(1)$, 
see \cite[Algorithm 6.3]{KirschmerVoight}. The test 
for a lattice minimum is much faster than the computation of an isometry.
\end{enumerate}
\end{remark}

\begin{example} \label{q15}
Unimodular lattices over $\Z[\sqrt{15}] $.
 As an example we take $K=\Q[\sqrt{15}]$. 
 Then $\Z_K = \Z[\sqrt{15} ]$, $h_K=2$, $h_K^+ = 4$. 
 The narrow class group of $K$ is represented by
 $$\{ [ \Z_K ] , [ (3,\sqrt{15} ) = \frakp _3 ] , [(5,\sqrt{15}) = \frakp _5 ] ,
 [(\sqrt{15} ) = \frakp_3\frakp_5 ] \}$$
 and the fundamental unit $\epsilon = 4+\sqrt{15}$ of $\Z_K$ is totally 
 positive. 

 We take $Q = \left( \frac{-1,-1}{K} \right) $ to be the quaternion algebra
 over $K$ ramified only at the two infinite places. 
 With the algorithm from \cite{KirschmerVoight} that is implemented in 
 Magma~\cite{Magma} we compute that $Q$ has 8 maximal orders 
each of class number 8. 
We list these maximal orders ${\mathcal M}_i$ $(1\leq i \leq 8)$  by giving  the structure of their unit group:

 $$\begin{array}{|c|c|c|c|c|} 
	 \hline
	 i &   {\mathcal M}_i^*/\Z_K^* &  {\mathcal M}_i^{(1)}/\{\pm 1\}  
	 & n(\mathcal{M}_i^*) / (\Z_K^*)^2 & 
	  n(N_i) / (\Z_K^*)^2   \\
	 \hline
 1 & C_2\times C_2 & C_2\times C_2 & 1 & \langle 2 \rangle  \\
 2 & C_2\times C_2 & C_2 & \langle \epsilon \rangle  & \langle 2 ,2\epsilon   \rangle  \\
 3 & A_4 & A_4 & 1 & \langle 2 \rangle  \\
 4 & C_2 & 1 & \langle \epsilon \rangle & \langle 2 ,2\epsilon   \rangle  \\
 5 & S_3& C_3 & \langle \epsilon \rangle  & \langle 2, 2\epsilon  \rangle  \\
 6 & S_3 & S_3 & 1 & \langle 2\epsilon \rangle  \\
 7 & C_2\times C_2 &  C_2 & \langle \epsilon \rangle  & \langle 2,2\epsilon \rangle  \\
 8 & C_3& C_3 & 1 & \langle 2\epsilon \rangle  \\ \hline
  \end{array}
  $$
From this information we get 
$$\Pi _i = \langle (2) \rangle , f_i  = f_{ij} = 1 \mbox{ for all } i,j $$ 
and $z_{ij} =1$ if $\{i,j\} \in \{ \{1,1\},\{3,3\}, \{1,3\} ,\{6,6\},\{8,8\},\{6,8\}\} $ and
$z_{ij} =0$ in all other cases.
We compute that 
$$[ n({\mathcal M}_1 {\mathcal M}_j ) ] =
\left\{ \begin{array}{ccc} [\Z_K ] & \mbox{ for }  & j\in \{1,7\} \\
 \ [\frakp_3 ] & \mbox{ for }  & j\in \{4\} \\ 
 \ [\frakp_5 ] & \mbox{ for }  & j\in \{2,3,6\} \\ 
\ [\frakp_3 \frakp_5 ] & \mbox{ for }  & j\in \{5,8 \} . \end{array} \right.
$$
As the class number is equal to the type number, all 
normal ideals are equivalent to ${\mathcal M}_i {\mathcal M}_{j}  $ 
for some $1\leq i,j\leq 8 $. Moreover 
$ [ n({\mathcal M}_i {\mathcal M}_{j} ) ] = 
[n({\mathcal M}_1 {\mathcal M}_i )]  \cdot  [n({\mathcal M}_1 {\mathcal M}_j )] $ can be computed from the information above. 
Using the information on $z_{ij}$ given before, Proposition \ref{latJ} 
now allows to deduce the number 
 of proper isometry
classes of $\Z_K$-lattices in each of the four genera as listed in the
next table. The columns are headed by a set of  indices $i$ whereas the
entries in the table give the set of values of $j$ such that 
$ n({\mathcal M}_i {\mathcal M}_{j}) $ lies in the 
narrow ideal class of the respective row. 
The entries below the $\# $ gives the number of proper isometry classes
of lattices obtained by these values $(i,j)$. Summing up these
entries in each row gives the proper class number $h^+$ of the genus as 
displayed in the first column of the table: 

$$\begin{array}{|ccr|cc|cc|cc|cc|} 
	\hline
	h^+  &  \fraka &   &  \{ 1,7 \} & \# & \{ 2,3,6 \} & \# & \{ 4 \} & \# & \{ 5,8 \}& \#  \\ 
	\hline
	22 & \Z_K  & &  \{ 1,7 \} & 5 & \{ 2,3,6 \} & 11 & \{ 4 \} & 1 & \{ 5,8 \}  & 5 \\ 
	\hline
	18 & \frakp _3^{-1} & & \{ 4 \} & 2 & \{ 5,8 \} & 7 & \{ 1,7 \} & 2 & \{ 2,3,6 \} & 7 \\
	\hline 
	18 & \frakp_5 ^{-1} & & \{ 2,3,6 \} & 7 & \{ 1,7  \}& 7 & \{ 5,8 \} & 2 & \{ 4 \} & 2 \\
	\hline 
	14 & \frakp_{15} ^{-1} & & \{ 5,8 \} & 4 & \{ 4 \} & 3 & \{ 2,3,6 \} & 3 & \{ 1,7 \} & 4 \\
	\hline 
\end{array} 
$$

For the four genera considered above, the trace lattices lie in the genera of 
even 15-modular (+ type) (see \cite{SchaSchuPi} for basic facts on modular lattices),
5-modular, 3-modular resp. unimodular lattices 
of dimension 8. Of course the latter 14 lattices are as $\Z $-lattices all
isometric to the $\E _8$-lattice, the unique positive definite even unimodular $\Z $-lattice of dimension 8. 
One finds 2 extremal 15-modular lattices (minimum $6$ as $\Z $-lattices): 
$({\mathcal M}_3, \epsilon n) $ and 
$({\mathcal M}_3 {\mathcal M}_6 , n)  \cong 
({\mathcal M}_6 {\mathcal M}_3 , n)  $. 
There is a unique extremal even 5-modular lattice of dimension 8 
(minimum 4 as $\Z $-lattice), 
so all the  $\Z $-trace lattices in ${\mathcal G}_{\frakp _3^{-1}} (Q,n)$ 
of minimum 4 are isometric to this lattice. 
These are 
$({\mathcal M}_i{\mathcal M}_j ,n )$ for 
$\{i,j\}  = \{ 2,5\} ,\{3,8 \}$ or 
$({\mathcal M}_i{\mathcal M}_j ,\epsilon n )$ for 
$\{i,j\}  = \{ 6,8 \}$.
\end{example}

\begin{example}\label{larger} 
  Let $K = \Q(\zeta_{19} + \zeta_{19}^{-1})$ be the totally real subfield of the cyclotomic field $\Q(\zeta_{19})$.
  Then there exists a unique prime ideal $\frakp$ of $\Z_K$ over $19$ and the narrow class group $\CL^+(K)$ is trivial.
  Let $Q = \left( \frac{-1, -19}{K} \right)$ be the quaternion algebra over $K$ ramified only at the infinite places and at $\frakp$.
  We implemented Algorithm \ref{algo} in Magma and applied it to compute the maximal integral $\Z_K$-lattices in $(Q, n)$.
  The timings below were done on an Intel Core i7 7700K.

  The computation of a maximal order in $Q$ took less than a second.
  For the second step, \cite[Algorithm 7.10]{KirschmerVoight} took about 9 minutes.
  It turns out that there are $t=185$ conjugacy classes of maximal orders in $Q$.
  Each of them has $h=356$ left equivalence classes of right ideals.
  The computation of the products in step (3) took $12$ minutes and it took another $50$ minutes to enumerate the orbit representatives in step (4).
  It turns out that $\bigcup_i T_i$ consists of $63466$ lattices.
  Since $\Z_{K, >0}^* = (\Z_K^*)^2$, we can always choose $u=1$ in the last step of the algorithm.
  So there are $63466$ proper isometry classes of  maximal integral $\Z_K$-lattices in $(Q, n)$. 
  The complete enumeration only took $71$ minutes.
  Enumerating such a large genus with Kneser's neighbour method would take several days.
\end{example}

\bibliography{quaternary}

\end{document}